\documentclass{article}
\usepackage[utf8]{inputenc}
\usepackage{amsmath,amsfonts,amssymb,amsthm,fancyhdr,bm,verbatim,hyperref,centernot, enumerate, bbm}
\usepackage{fullpage}
\usepackage{mathtools}
\usepackage{todonotes}

\newtheorem{theorem}{Theorem}
\newtheorem{lemma}[theorem]{Lemma}

\newtheorem{proposition}[theorem]{Proposition}
\theoremstyle{definition}
\newtheorem{definition}[theorem]{Definition}
\DeclareMathOperator{\LCS}{\mathsf{LCS}}
\DeclareMathOperator{\SCS}{\mathsf{SCS}}

\DeclareMathOperator{\del}{\mathsf{del}}
\DeclareMathOperator{\ins}{\mathsf{ins}}

\newcommand{\ml}{m_{\mathsf{LCS}}}
\newcommand{\ms}{m_{\mathsf{SCS}}}
\newcommand{\Left}{\mathsf{left}}
\newcommand{\floor}[1]{\left\lfloor #1\right\rfloor}
\newcommand{\dist}{\textrm{d}}

\title{Logarithmically larger deletion codes of all distances}
\author{Noga Alon \thanks{Department of Mathematics, Princeton University, Princeton, NJ 08544, USA and Schools of Mathematical Sciences and Computer Science, Tel Aviv University, Tel Aviv, Israel. Email: {nalon@math.princeton.edu}. Research supported in part by NSF grant DMS-2154082}
\and 
Gabriela Bourla \thanks{Department of Mathematics, Princeton University, Princeton, NJ 08544. Email: gbourla@princeton.edu. Research supported by the math department's undergraduate funding.}
\and 
Ben Graham \thanks{Department of Mathematics, Princeton University, Princeton, NJ 08544. Email: bagraham@princeton.edu. Research supported by the math department's undergraduate funding.} 
\and 
Xiaoyu He \thanks{Department of Mathematics, Princeton University, Princeton, NJ 08544. Email: xiaoyuh@princeton.edu. Research supported by NSF Award DMS-2103154.}
\and 
Noah Kravitz \thanks{Department of Mathematics, Princeton University, Princeton, NJ 08544. Email: nkravitz@princeton.edu. Research supported by NSF GRFP Award DGE-2039656.}
}

\begin{document}

\maketitle

\begin{abstract}
The deletion distance between two binary words $u,v \in \{0,1\}^n$ is the smallest $k$ such that $u$ and $v$ share a common subsequence of length $n-k$. A set $C$ of binary words of length $n$ is called a $k$-deletion code if every pair of distinct words in $C$ has deletion distance greater than $k$. In 1965, Levenshtein initiated the study of deletion codes by showing that, for $k\ge 1$ fixed and $n$ going to infinity, a $k$-deletion code $C\subseteq \{0,1\}^n$ of maximum size satisfies $\Omega_k(2^n/n^{2k}) \leq |C| \leq O_k( 2^n/n^k)$.
We make the first asymptotic improvement to these bounds by showing that there exist $k$-deletion codes with size at least $\Omega_k(2^n \log n/n^{2k})$. Our proof is inspired by Jiang and Vardy's improvement to the classical Gilbert--Varshamov bounds.  We also establish several related results on the number of longest common subsequences and shortest common supersequences of a pair of words with given length and deletion distance.
\end{abstract}

\section{Introduction}
The main goal of coding theory is to construct schemes for efficiently and faithfully communicating messages across a noisy channel.  In this paper, we study a noise model proposed by Levenshtein~\cite{Le} in which messages are finite binary words in $\{0,1\}^n$ and the communication channel, a ``deletion channel,'' deletes a fixed number $k$ of bits from the transmitted message; the locations of the deletions are unknown to the receiver. Deletion errors are a special case of ``synchronization errors'', which are remarkably poorly understood compared to the better-studied noise models of bit flips and bit erasures.

Formally, for $n\ge k\ge 1$ a \emph{$k$-deletion code of length $n$} is a collection $C\subseteq \{0,1\}^n$ of binary words with the property that for any $y\in \{0,1\}^{n-k}$, there is at most one $x\in C$ containing $y$ as a subsequence. Equivalently, $C$ is a $k$-deletion code if for all distinct $s,t\in C$, the longest common subsequence of $s$ and $t$ has length strictly smaller than $n-k$. We would like to determine the maximum size $D(n,k)$ of a $k$-deletion code of length $n$. In his seminal 1965 paper~\cite{Le}, Levenshtein established the upper and lower bounds

\begin{equation}\label{eqn:levenshtein}
\Omega_k\left(\frac{2^n}{n^{2k}} \right) \leq D(n,k) \leq O_k\left(\frac{2^n}{n^{k}} \right).
\end{equation}

For the case $k=1$, Levenshtein used a construction of Varshamov and Tenengolts~\cite{VaTe} to show that $D(n,1)=\Theta(2^n/n)$, so the upper bound in (\ref{eqn:levenshtein}) is asymptotically correct in this case. In contrast, despite a great deal of effort on many related questions in recent years, neither bound in \eqref{eqn:levenshtein} has been improved for any fixed $k\ge 2$.  Since the breakthrough in~\cite{GuHa}, there has been some progress on constructing explicit $2$-deletion codes that nearly match Levenshtein's lower bound; see also \cite{BrGuZb, BuGuHa, GuHeLi, SiBr}.

Our main result is a logarithmic improvement on the lower bound, which holds for all alphabet sizes. Write $D_\alpha(n,k)$ for the maximum size of a $k$-deletion code of length $n$ over the fixed alphabet $[\alpha]$.\footnote{When $\alpha=2$, we will sometimes instead work over the ``usual'' binary alphabet $\{0,1\}$; this will not cause any confusion.}

\begin{theorem}\label{thm:main}
If $n \ge k \ge 2$ and $\alpha\ge 2$, then $D_\alpha(n,k)\ge \Omega_{\alpha, k}(\alpha^n \log  n/n^{2k})$.
\end{theorem}

Our proof is nonconstructive: We reduce the problem of finding large codes to the problem of finding a large independent set in the associated \emph{$k$-deletion graph} $\Gamma_{n,k,\alpha}$.  The graph $\Gamma_{n,k,\alpha}$ has vertex set $[\alpha]^n$, and two words are connected by an edge if they have a common subsequence of length 
at least $n-k$. We show that $\Gamma_{n,k,\alpha}$ is locally sparse, that is, contains few triangles.  Theorem~\ref{thm:main} then follows from standard lemmas about the independence number of locally sparse graphs. A similar application of local sparsity to coding theory appears in the work of Jiang and Vardy~\cite{JiVa}, who obtained the first asymptotic improvements on the Gilbert--Varshamov bounds.

Counting triangles is harder for the graph $\Gamma_{n,k,\alpha}$ than it is for the extremely symmetric setting studied by Jiang and Vardy, where the analogous graph is just a power of the Hamming cube. In contrast, $\Gamma_{n,k,\alpha}$ is not even regular. In order to overcome these difficulties, we restrict our attention to ``pseudorandom'' words in the graph that are related by ``pseudorandom'' sequences of insertion and deletion operations, for suitable notions of pseudorandomness.

In this work we also prove additional results about the number of longest common subsequences and shortest common supersequences of a pair of words, as a function of their lengths and deletion distance. These bounds, which were necessary in earlier versions of our proof of Theorem~\ref{thm:main}, are of independent interest and may be useful for future study of deletion codes and the structure of the graphs $\Gamma_{n,k,\alpha}$.

We denote the length of a word $u$ by $|u|$.  We say that the word $w$ is a \emph{subsequence} of the word $u$ if $w$ can be obtained from $u$ by deleting some of the letters of $u$.  If $w$ is a subsequence of the words $u$ and $v$, then we say that $w$ is a \emph{common subsequence} of $u$ and $v$; further, $w$ is a \emph{longest common subsequence} (or \emph{LCS}) of $u$ and $v$ if it is a common subsequence of maximum length.  We let $\LCS(u,v)$ denote the length of an LCS of $u$ and $v$.  If $u$ and $v$ are words of the same length $|u|=|v|=n$, then we define the \emph{deletion distance} between $u$ and $v$ to be $\dist(u,v):=n-\LCS(u,v)$.  One can define \emph{shortest common supersequences} (or \emph{SCS}'s) and the \emph{insertion distance} analogously.  It is well known that $\LCS(u,v)+\SCS(u,v)=|u|+|v|$ for all words $u,v$, so, in particular, deletion distance and insertion distance are identical; LCS's and SCS's are in this sense dual.

For words $u$ and $v$, we define the \emph{LCS multiplicity} $\ml(u,v)$ (respectively, \emph{SCS multiplicity} $\ms(u,v)$) to be the number of distinct LCS's (respectively, SCS's) of $u$ and $v$.  The following simple inequality relating LCS and SCS multiplicity is probably known to experts, but we could not locate a reference in the literature.

\begin{proposition}\label{prop:multiplicity-ineq}
For all words $u,v$, we have $\ml(u,v)\leq \ms(u,v)$.
\end{proposition}

Our main result on LCS and SCS multiplicity is the following.

\begin{theorem}\label{thm:multiplicity-upper-bound}
Let $n, a, b$ be natural numbers with $n \geq a+b$.  If $u$ and $v$ are words with lengths $n-a$ and $n-b$ (respectively) and $\SCS(u,v) = n$ (equivalently, $\LCS(u,v) =n-a-b$), then $$\ml(u,v) \leq \ms(u,v) \leq \binom{a+b}{a}.$$
\end{theorem}

 The $a=b$ case can be phrased symmetrically as follows: If $u,v$ are words of equal length with $\dist(u,v)=d$, then we have $$\ml(u,v)\leq \ms(u,v) \leq \binom{2d}{d},$$ independent of the lengths of $u,v$.  We also prove in the appendix that this theorem is tight in that for all choices of $a$ and $b$ and all sufficiently large $n$ (in terms of $a,b$), there exists a pair of words $u,v$ for which equality is attained in both inequalities.

The paper is organized as follows. We prove the main result Theorem~\ref{thm:main} in Section~\ref{sec:main}; we prove Proposition~\ref{prop:multiplicity-ineq} and Theorem~\ref{thm:multiplicity-upper-bound} in Section~\ref{sec:upper-extremal}; finally, we describe a family of pairs of words which attain equality in Theorem~\ref{thm:multiplicity-upper-bound} in the appendix.

We use standard asymptotic notation, as follows.  If $f(n), g(n): \mathbb{N}\to \mathbb{R}$ are functions, then we write $f=O(g)$ to indicate that there is some constant $C>0$ such that $|f(n)|\leq Cg(n)$ for all natural numbers $n$.  If $g$ is nonnegative, then we write $f=\Omega(g)$ to indicate that $g=O(f)$.  We write $f=\Theta(g)$ if $f=O(g)$ and $g=O(f)$.  Subscripts on $O, \Omega, \Theta$ indicate that the implied constants $C$ may depend on the subscripted parameters. All logarithms are base-$2$.

\section{Proof of Theorem~\ref{thm:main}} \label{sec:main}

In this section we prove Theorem~\ref{thm:main} in two steps: We reduce the problem to counting triangles in the $k$-deletion graph $\Gamma_{n,k,\alpha}$, and then we approximate this triangle count.  Observe that $D_\alpha(n,k)$ is by definition the independence number of $\Gamma_{n,k,\alpha}$. We need the following standard lemma of Ajtai, Koml\'os, and Szemer\'edi~\cite{AjKoSz} on independence numbers of graphs with few triangles.  See also Shearer~\cite{Sh} for a simpler argument and \cite[pp. 336-337]{AS} for a very short proof. This line of work has led to several important developments in extremal graph theory and Ramsey theory.

\begin{lemma}[\cite{AjKoSz}, Lemma 5]\label{lem:sparse-neighborhoods}
For any $\varepsilon>0$ and any graph $G$ on $N\ge 1$ vertices with average degree $d$ containing $T < N d^{2-\varepsilon}$ triangles, we have $\alpha(G) \ge \Omega_\varepsilon((N/d)\log d).$
\end{lemma} 

It follows easily that we can replace average degree $d$ by maximum degree $\Delta$ in the lemma above, and this is the form we will use. The graph $\Gamma_{n,k,\alpha}$ has $N=[\alpha]^n$ vertices and maximum degree $\Delta = O_{\alpha, k}(n^{2k})$, since from any given vertex $u\in [\alpha]^n$, a neighbor $v$ can be obtained by choosing $k$ letters of $u$ to delete in at most $\binom{n}{k}$ ways and then $k$ letters to insert in at most $\binom{n}{k} \alpha^k$ ways. Thus, if we want to use Lemma~\ref{lem:sparse-neighborhoods} to prove that $D_\alpha(n,k) =\Omega_{\alpha,k}(\alpha^n \log n/n^{2k})$, it suffices to show that the number of triangles in $\Gamma_{n,k,\alpha}$ is $O_{\alpha, k}(\alpha^n n^{4k-\varepsilon})$ for some $\varepsilon >0$.

We will actually prove the sharper bound that $\Gamma_{n,k,\alpha}$ has $O_{\alpha, k}(\alpha^n n^{3k} (\log n)^k)$ triangles.  This estimate is tight up to the logarithmic factor. It will be convenient to focus our attention on ``pseudorandom'' words, as follows. 
If $u \in [\alpha]^n$ is a word of length $n$ and $S \subseteq [n]$ is a subset, 
let $u_S$ denote the subword of $u$ indexed by $S$.  
If $I=[x,y]$ is an interval, then we call $u_{I}$ a \emph{subinterval} of
$u$.  For $1 \leq \lambda \leq n$, we say that $u\in[\alpha]^n$ is \emph{$\lambda$-nonrepeating} if $u_I \neq u_{J}$ for all pairs of distinct intervals $I,J \subseteq [n]$ of length $\lambda$; $u$ is \emph{$\lambda$-repeating} otherwise. By the first-moment method, if $\lambda > (2+\varepsilon) \log n$ for some $\varepsilon > 0$, then almost all words of length $n$ are $\lambda$-nonrepeating (see the proof of Lemma~\ref{lem:triangles} for the formal proof of this fact).

Next we introduce notation for a sequence of insertion and deletion operations. Let $u\in [\alpha]^n$, $t\in \{\del,\ins_1,\ins_2,\ldots, \ins_\alpha\}$ and $i \in [0,n]$, where $i$ is not allowed to be $0$ if $t=\del$ (since the $0$-th letter of $u$, which does not exist, cannot be deleted). We write $f_{i, t}(u)$ for the word obtained from $u$ by deleting $u_i$, if $t = \del$, inserting an $x$ after $u_i$, if $t = \ins_x$. Here, ``inserting after $u_0$'' means inserting before $u_1$.

\begin{definition}
Fix nonnegative integers $n$ and $\ell$.  Let $I = (i_\ell, i_{\ell-1},\ldots, i_1)$ be a nonincreasing sequence of nonnegative integers $n \ge i_\ell \ge i_{\ell-1} \ge \ldots \ge i_1 \ge 0 $, and let $T = (t_\ell, \ldots, t_1) \in \{\del,\ins_1,\ins_2,\ldots, \ins_\alpha\}^\ell$ be a sequence of insertion/deletion types.  We further 
require that if $t_j=\del$ then $i_j \ne 0$ 
(the $0$-th letter of $u$ cannot be deleted) and $i_{j-1} < i_j$ (we do not operate on an already-deleted letter).  We then call the pair $(I,T)$ a \emph{sequence of $\ell$ insertions and deletions}, and we write
$$f_{I,T}(u):=(f_{i_1,t_1} \circ f_{i_2,u_2}\circ \cdots \circ f_{i_\ell, t_\ell})(u)$$
for the composition of the operations $f_{i_\ell, t_\ell}$ through $f_{i_1,t_1}$ applied to a word $u\in [\alpha]^n$.
\end{definition}

Whenever one obtains a word $v$ from $u$ by inserting and deleting letters, one can reorder these operations to find a sequence $(I,T)$ of insertions and deletions such that $v=f_{I,T}(u)$. Note that, because the elements of $I$ are nonincreasing, an earlier operation cannot shift the location of a later operation. In particular, $i_j$ is not only the position in $(f_{i_{j+1},t_{j+1}}\circ \cdots \circ f_{i_\ell, t_\ell})(u)$ at which the operation $f_{i_j,t_j}$ is applied, but also the original position in $u$ at which the operation occurs. This lets us refer unambiguously to the ``position'' $i_j$ in $u$ of each operation $f_{i_j, t_j}$.

We say that an element $i$ of a set $I\subseteq [0,n]$ is \emph{$\lambda$-isolated} if $\lambda<i<n-\lambda$ and no other element $j\in I$ satisfies $|j-i|\le 2\lambda$.  We are now ready to prove our key lemma.

\begin{lemma}\label{lem:isolated}
Let $n,k, \lambda \ge 1$, and let $u,v\in [\alpha]^n$ be $\lambda$-nonrepeating words such that $v=f_{I,T}(u)$ for some sequence of operations $(I,T)$.  If the number of $\lambda$-isolated elements of $I$ is at least $2k+1$, then $\dist(u,v) > k$.
\end{lemma}

\begin{proof}
We may pick $2k+1$ of the $\lambda$-isolated terms of $I$ and call them $j_{2k+1}>j_{2k}>\cdots > j_1$; let the corresponding terms of $T$ be $t_{2k+1},\ldots, t_1$. Note that since the operations of $(I,T)$ are applied in decreasing order of index, these operations $f_{j_s,t_s}$ are applied in decreasing order of $s$ as well. For each $\lambda < j < n-\lambda$, write $L(j):= [j-\lambda, j-1]$ and $R(j):= [j+1, j+\lambda]$ for the length-$\lambda$ intervals in $[n]$ immediately to the left and right of $j$. The definition of $\lambda$-isolation implies that the $4k+2$ intervals 
\[
L(j_1),R(j_1),L(j_2),R(j_2),\ldots,L(j_{2k+1}),R(j_{2k+1})
\]
are pairwise disjoint intervals of length $\lambda$. Moreover, no insertion or deletion operations occur in any of the corresponding $4k+2$ subintervals of $u$ (since the $\lambda$-isolation assumption ensures that the $j_i$'s are far apart). Since $u$ and $v$ are $\lambda$-nonrepeating, each of these $4k+2$ words appears exactly once as a subinterval of $u$ and once as a subinterval of $v$.

The key observation is that when we apply $f_{j_s,t_s}$, we either insert or delete a single letter between $u_{L(j_s)}$ and $u_{R(j_s)}$. Inside $u$, these two subintervals appear with exactly one letter between them, and no other insertion or deletion operations happen nearby. Thus, $u_{L(j_s)}$ and $u_{R(j_s)}$ appear in $v$, and the number of letters in $v$ between the unique appearances of $u_{L(j_s)}$ and $u_{R(j_s)}$ is either $0$ (if a letter was deleted by $f_{j_s,t_s}$) or $2$ (if a letter was inserted by $f_{j_s,t_s}$).

Assume for the sake of contradiction that $\dist(u,v) \le k$, and let $(I',T')$ be a sequence of at most $k$ insertions and at most $k$ deletions such that $v = f_{ I', T'}(u)$. Since $|I'| \le 2k$, there exists some $1\le s\le 2k+1$ for which $I'$ is disjoint from the entire length-$(2\lambda+1)$ subinterval $[j_s-\lambda, j_s+\lambda]$. It follows that $u_{[j_s-\lambda, j_s+\lambda]}$ appears unaltered as a subinterval of $v$, and in particular the unique copy of $u_{L(j_s)}$ in $v$ and the unique copy of $u_{R(j_s)}$ in $v$ have exactly one letter between them. This contradicts the key observation in the previous paragraph, so we conclude that $\dist(u,v) > k$, as desired.
\end{proof}

The next lemma lets us upper-bound the number of triangles in $\Gamma_{n,k,\alpha}$ and, more generally, the number of triples $(u,v,w) \in ([\alpha]^n)^3$ with prescribed values of $\dist(u,v), \dist(v,w), \dist(w,u)$.

\begin{lemma}\label{lem:triangles}
Let $n\ge a \ge b\ge c \ge 1$.  The number of triples $(u,v,w) \in ([\alpha]^n)^3$ with $\dist(u,v) \le a$, $\dist(v,w) \le b$, and $\dist(w,u) \le c$ is $O_{a,\alpha}(\alpha^n n^{a+b+c} (\log n)^{b+c-a})$.
\end{lemma}

\begin{proof}
Say that a triple $(u,v,w) \in ([\alpha]^n)^3$ is \emph{good} if $\dist(u,v) \le a$, $\dist(v,w) \le b$, and $\dist(w,u) \le c$.  Note that $\dist(u,v) \leq \dist(v,w)+\dist(w,u)\leq b+c$ by the Triangle Inequality, so all good triples $(u,v,w)$ satisfy $\dist(u,v) \leq b+c$, and we may restrict our attention to the regime $a \leq b+c$.

Let $\lambda = 10a\log n$, and observe that the probability of a uniformly random $u\in [\alpha]^n$ being $\lambda$-repeating is at most $\binom{n}{2}n^{-10a} \le n^{-8a}$. Thus, the total number of such exceptional words is at most $\alpha^n n^{-8a}$. For each $u\in [\alpha]^n$, there are at most $O_a(n^{2a})$ words $v$ at distance at most $a$ and at most $O_a(n^{2c})$ words $w$ at distance at 
most $c$, so there are at most $O_a(n^{2a + 2c}) \le O_a(n^{4a})$ good triples $(u,v,w)$ for each choice of fixed $u$ (and likewise for each fixed choice of $v$ or $w$). We find that the total number of good triples containing a $\lambda$-repeating word is $\alpha^n n^{-8a}\cdot O_a(n^{4a}) = o(\alpha^n)$, which is negligible. It remains to bound the number of good triples consisting of $\lambda$-nonrepeating words.

It suffices to prove that every $\lambda$-nonrepeating $u$ lies in at most $O_a(n^{a+b+c}(\log n)^{a+b-c})$ good triples $(u,v,w)$ with $v,w$ both $\lambda$-nonrepeating. Note that a good triple $(u,v,w)$ is uniquely determined by the data of $u$ and sequences $(I,T)$, $(I',T')$ of insertion or deletion operations for which $w=f_{I,T}(u)$ and $v=f_{I',T'}(w)$. Since $\dist(u,w) \le c$ and $\dist(w,v) \le b$, we may choose $(I,T)$ to have length at most $2c$ and $(I',T')$ to have length at most $2b$.  Since $f_{I', T'}(f_{I,T}(u))=v$, we can ``combine'' the insertions and deletions of $(I,T)$ and $(I', T')$ to obtain a sequence $(I'', T'')$ of insertions and deletions of length $|I''|=|I|+|I'| \leq 2b+2c$ such that $f_{I'', T''}(u)=v$.  Furthermore, there are only $O_a(1)$ choices of $(I,T)$ and $(I',T')$ that produce each such sequence $(I'', T'')$. Thus, for a given $u$, the number of good triples $(u,v,w)$ is at most $O_a(1)$ times the number of ways to pick a sequence $(I'',T'')$ of at most $2b+2c$ total insertions and deletions such that $v = f_{I'',T''}(u)$ is $\lambda$-nonrepeating and $\dist(u,v) \le a$.

By Lemma~\ref{lem:isolated}, the assumption $\dist(u,v) \le a$  implies that at most $2a$ of the elements of $I''$ are $\lambda$-isolated. We claim that the total number of ways to pick such an $I''$ is at most $O_a(n^{a+b+c}(\log n)^{b+c-a})$.  Indeed, we can define an equivalence relation $\sim$ on the elements of $I''$ by setting $i\sim j$ if $|i-j|\leq 2 \lambda$ and then taking the transitive closure.  Let $Q$ denote the number of equivalence classes. There are at most $2a$ equivalence classes of size $1$ coming from $\lambda$-isolated elements, and $y\le 2$ equivalence classes of size $1$ coming from elements $i\in I''$ satisfying $i \le \lambda$ or $i\ge n-\lambda$, which we call \textit{boundary} equivalence classes. Hence, altogether $Q \leq 2a+y+(2b+2c-2a-y)/2=a+b+c+y/2$.  There are at most $n^{Q-y}$ ways to choose the minimal elements of the non-boundary equivalence classes, $\lambda^y$ ways to choose the minimal elements of the boundary equivalence classes, and then $(2\lambda)^{2b+2c-Q}$ ways to choose the remaining elements of $I''$. The quantity $n^{Q-y}(2\lambda)^{y+2b+2c-Q}$ is at most $n^{a+b+c}(2\lambda)^{b+c-a}$, and multiplying by $a+b+c=O_a(1)$ (for the possible values of $Q$ and $y$) establishes the claim. Finally, there are at most $(\alpha+1)^{2b+2c}=O_{a,\alpha}(1)$ ways to 
pick $T''$, and this completes the proof. 
\end{proof}

The proof of Theorem~\ref{thm:main} is now immediate.

\begin{proof}[Proof of Theorem~\ref{thm:main}]
By Lemma~\ref{lem:triangles} with $a=b=c=k$, the number $T$ of triangles in $\Gamma_{n,k,\alpha}$ satisfies $T=O_{\alpha, k}(\alpha^n n^{3k} (\log n)^k)$. Applying Lemma~\ref{lem:sparse-neighborhoods} with $N = \alpha^n$ and $\Delta = O_{\alpha, k} (n^{2k})$, we find that
\[
D(n,k) =\Omega_{\alpha, k}(\alpha^n \log n/n^{2k}),
\]
as desired.
\end{proof}

\section{LCS and SCS Multiplicity} \label{sec:upper-extremal}

In this section, we prove Proposition~\ref{prop:multiplicity-ineq} and Theorem~\ref{thm:multiplicity-upper-bound}.  Before proving Proposition~\ref{prop:multiplicity-ineq}, which says that the LCS multiplicity is always smaller than or equal to the SCS multiplicity, we set up one piece of notation.  Suppose $u$ is a word of length $n$ which contains the word $w$ of length $\ell$ as a subsequence.  Then there is at least one subset $S \subseteq [n]$ of size $\ell$ such that $u_S=w$, and there may be several such subsets.  We define $\Left(u,w)$ to be the smallest of these subsets according to the lexicographic ordering; that is, we choose $S$ to have the smallest possible smallest element, and we break ties by looking at the second-smallest element, and so on.  We can think of $\Left(u,w)$ as describing the position of the ``left-most'' copy of $w$ in $u$.

\begin{proof}[Proof of Proposition~\ref{prop:multiplicity-ineq}]
Let $|u|=m$, $|v|=n$, and $\LCS(u,v)=\ell$, and note that $\SCS(u,v)=m+n-\ell$.  We define an injective map $\varphi$ from the set of LCS's of $u,v$ 
to the set of SCS's of $u,v$, as follows.  Fix an LCS $w$ of $u,v$.  We now construct an SCS $y$ of $u,v$ one letter at a time. 

To illustrate the idea, consider $u=1011$ and $v=0101$. There are two choices of an LCS for $u$ and $v$, namely, $w=101$ and $w=011$. For the first choice $w=101$, we mark its left-most copy in $u = \underline{101}1$ and $v = 0\underline{101}$, and ``glue'' $u$ and $v$ together along these copies to obtain $\varphi(w) = 0\underline{101}1$. For the second choice $w=011$, we mark its left-most copy in $u=1\underline{011}$ and $v=\underline{01}0\underline{1}$ and glue to obtain $\varphi(w) = 1\underline{01}0\underline{1}$.

Here is the formal description of the algorithm. To begin, initialize two indices $i=j=1$ to track our current indices in $u$ and $v$, respectively.  For each $1 \leq k \leq m+n-\ell$, define $y_k$ according to the following algorithm: \begin{enumerate}[(i)]
    \item If $i \notin \Left(u,w) \cup \{m+1\}$, then let $y_k=u_i$ and increment $i$. \label{case:1}
    \item If $i \in \Left(u,w) \cup \{m+1\}$ and $j \notin \Left(v,w) \cup \{n+1\}$, then let $y_k = v_j$ and increment $j$. \label{case:2}
    \item If $i \in \Left(u,w)$ and $j \in \Left(v,w)$, then let $y_k = u_i$ (which is also equal to $v_j$), and increment both $i$ and $j$. (An easy induction shows that the $r$-th time this third possibility occurs, we have $u_i=v_j=w_r$.)\label{case:3}
\end{enumerate}
The number of times the algorithm falls into cases (\ref{case:1}), (\ref{case:2}), (\ref{case:3}) above are (respectively) $m-\ell$, $n-\ell$, and $\ell$, so the algorithm terminates at exactly $i=m+1$, $j=n+1$, with $m+n-\ell$ well-defined letters $y_1,\ldots, y_{m+n-\ell}$. Define $y\coloneqq y_1y_2\cdots y_{m+n-\ell}$. We have $|y|=m+n-\ell$ and $y$ contains $u,v$ as subsequences, so $y$ is in fact an SCS of $u,v$. Finally, let $\varphi(w)=y$.

It remains to show that $\varphi$ is injective, i.e., that $w$ can be recovered from $\varphi(w)$.  Note that item (iii) occurs if and only if $u_i=v_j$, so, working from $k=1$ to $k=m+n-\ell$, we can determine the set $K$ of indices $k$'s for which item (iii) occurs.  By the parenthetical remark in item (iii), we get $w=\varphi(w)_K$, as needed.
\end{proof}

As promised, we now prove Theorem~\ref{thm:multiplicity-upper-bound} on the sharp upper bound for SCS multiplicity (and by extension LCS multiplicity).

\begin{proof}[Proof of Theorem~\ref{thm:multiplicity-upper-bound}]
Due to Proposition~\ref{prop:multiplicity-ineq}, it suffices to prove that $m_{\SCS}(u,v) \leq \binom{a+b}{a}$.  We proceed by induction on $a+b$, and then by induction on $n$ for each fixed value of $a+b$.  The base case $a=b=0$ is trivial.  We now perform the (outer) induction step.  If $u,v$ have a common prefix, then every length-$n$ common supersequence of $u,v$ must also share this prefix; removing this common prefix preserves $a,b$ while reducing $n$, and we conclude by our (inner) induction hypothesis on $n$.  Suppose now that $u,v$ have different first letters.  The key observation is that every length-$n$ common supersequence of $u,v$ is of the form
$$u_1x \quad \text{or} \quad v_1y,$$
where $x$ is a length-$(n-1)$ common supersequence of $u_{[2,n-a]}$ and $v$ and $y$ is a length-$(n-1)$ common supersequence of $u$ and $v_{[2,n-b]}$.  If $\SCS(u_{[2,n-a]},v)=n-1$, then our (outer) induction hypothesis with $(n,a,b)$ replaced by $(n-1,a,b-1)$ gives that there are at most $\binom{a+b-1}{a}$ possibilities for $x$.  If instead $\SCS(u_{[2,n-a]},v)=n$, then there are no possibilities for $x$.  Likewise there are at most $\binom{a+b-1}{a-1}$ possibilities for $y$.  The result now follows from Pascal's Identity for binomial coefficients, namely, $\binom{j}{i}+\binom{j}{i+1}=\binom{j+1}{i+1}$ for natural numbers $0 \leq i \leq j-1$.
\end{proof}

We now show that the methods in Section~\ref{sec:main} can be used to prove that most words $u,v \in [\alpha]^n$ at a given distance have a unique SCS and LCS.

\begin{proposition}
If $n\ge k \geq 1$, then the number of pairs $u,v \in [\alpha]^n$ with $\dist(u,v)=k$ and $\ms(u,v)>1$ is $O_{\alpha, k}(\alpha^{n} n^{2k-1}\log n)$.
\end{proposition}
\begin{proof}
    If $u,v \in [\alpha]^n$ and $\dist(u,v) = k$, then there exists a sequence of $2k$ operations $(I,T)$ for which $v = f_{I,T}(u)$. There are at most $O_{\alpha, k}(n^{2k})$ choices of $(I,T)$ of length $2k$. We say that $(u,v)$ is \textit{exceptional} if $\ms(u,v)>1$.

    Let $\lambda = 3\log n$. The probability of a uniformly random $u\in [\alpha]^n$ being $\lambda$-repeating is at most $\binom{n}{2} 2^{-\lambda} \le n^{-1}$. The number of exceptional pairs $(u,v)$ for which either $u$ or $v$ is $\lambda$-repeating is thus at most $O_{\alpha,k}(\alpha^n n^{2k-1})$, so it remains to count exceptional pairs where both $u$ and $v$ are $\lambda$-nonrepeating.

    Suppose now that $I$ is $\lambda$-separated. For each $\lambda < j < n-\lambda$, write $I(j) \coloneqq [j-\lambda, j+\lambda]$, $L(j):= [j-\lambda, j-1]$, and $R(j):= [j+1, j+\lambda]$. Since $I$ is $\lambda$-separated, we can partition $u=z_0 u_{I(j_1)} z_1 u_{I(j_2)} \cdots z_{2k-1} u_{I(j_{2k})} z_{2k}$ into subintervals (where the $z_{s}$'s may be empty). Furthermore, since the operations in $(I,T)$ operate only within the $u_{I(j_s)}$'s, $v$ can also be partitioned into $v=z_0 v_{I(j_1)} z_1 v_{I(j_2)} \cdots z_{2k-1} v_{I(j_{2k})} z_{2k}$ where the intermediate subintervals $z_s$ are the same as in $u$. We now claim that the only SCS of $u$ and $v$ is the word $w = z_0 w_1 z_1\cdots w_{2k}z_{2k}$ where $w_s$ is the unique SCS of $u_{I(j_{s})}$ and $v_{I(j_{s})}$ (which differ by one letter).

    The $4k$ intervals
    \[
    L(j_1),R(j_1),L(j_2),R(j_2),\ldots,L(j_{2k}),R(j_{2k})
    \]
    are disjoint intervals of length $\lambda$. By the definitions of $u$, $v$, $I$, and $T$, the subintervals $u_{L(j_s)}$ and $u_{R(j_s)}$ each appear in $u$ exactly once, and for each $s$, $u_{L(j_s)}$ and $u_{R(j_s)}$ are separated by one letter $u_{j_s}$. In $v$, they also appear exactly once each, and for each $s$, $u_{L(j_s)}$ and $u_{R(j_s)}$ are separated by zero or two letters in $v$, depending on whether $f_{j_s, t_s}$ performs an insertion or a deletion.

    Let $w$ be a shortest common supersequence of $u$ and $v$, so the length of $w$ is $|u|+k$. We can form $w$ from $u$ by a sequence of $k$ insertions $(I'_w, T'_w)$, and $v$ from $w$ by a sequence of $k$ deletions $(I''_w,T''_w)$. In particular, the combined operation $(I_w,T_w)$ for which $f_{I_w, T_w} = f_{I''_w, T''_w}\circ f_{I'_w, T'_w}$ is a sequence of $2k$ operations for which $v = f_{I_w,T_w}(u)$. Using the observations in the previous paragraph, we can ``read off'' from the distance between the copies of $u_{L(j_s)}$ and $u_{R(j_s)}$ inside $v$ that $I_w$ must have exactly one element in each $I(j_s)$. Unwinding the definition of $(I_w,T_w)$, this means $w$ must be of the form $w = z_0 w_1 z_1\cdots w_{2k}z_{2k}$ where $w_s$ is the unique SCS of $u_{I(j_{s})}$ and $v_{I(j_{s})}$, as desired.

    This proves the claim and shows that $\ms(u,v)=1$ if $u$ and $v$ are $\lambda$-nonrepeating and $I$ is $\lambda$-separated. The number of choices of $I$ not $\lambda$-separated is at most $O_{\alpha, k}(n^{2k-1}\lambda) = O_{\alpha, k}(n^{2k-1}\log n)$, so the total number of exceptional pairs is at most $O_{\alpha, k}(\alpha^n n^{2k-1}\log n)$, as desired.
\end{proof}

\vspace{3mm}

\noindent {\bf Acknowledgments.} We are grateful to Venkatesan Guruswami for helpful conversations, and to the anonymous referees for comments that improved the presentation of this paper.  We thank Andrew Lin for pointing out an error in our original proof of Theorem~\ref{thm:multiplicity-upper-bound}.

{}

\appendix

\section*{Appendix}\label{sec:lcs-tightness}

In this appendix, we construct a family of pairs of words achieving equality in Theorem~\ref{thm:multiplicity-upper-bound}.  It is easy to find pairs of words achieving equality for the SCS bound.  For instance, we can take $u=0^a$, $v=1^b$; then $\SCS(u,v)=a+b$ and $\ms(u,v)=\binom{a+b}{a}$ since the SCS's of $u,v$ are precisely the words containing $a$ $0$'s and $b$ $1$'s.  To find longer words achieving equality with the same values of $a,b$, simply append a fixed word 
$w$ (for instance, $w=0^c$) to the right of both $u,v$.

It seems that there is no similarly simple example achieving equality for the LCS bound, and our construction requires a delicate induction.  If $u$ is a (nonempty) word of length $n$ and $m$ is a natural number, then we define $u^{\langle m \rangle}$ to be the prefix of length $m$ of the infinite word $uuu \cdots$.  For instance, $(01)^{\langle 
 7\rangle}=0101010$ and $(0110)^{\langle 3 \rangle}=011$.  Our extremal example is as follows.

\begin{proposition}
\label{prop:lower}
For every $c\ge1$, the words $u=(10)^{\langle 4c-2 \rangle}$, $v=(0110)^{\langle 4c-2\rangle}$ satisfy $\dist(u,v)=c$ and $\ml(u,v)=\binom{2c}{c}$.
\end{proposition}

Let us explain why Proposition~\ref{prop:lower} provides equality cases for Theorem~\ref{thm:multiplicity-upper-bound} for all choices of $a,b$.  Suppose $a_0,b_0$ are given, and consider the words $u,v$ produced by the $c=a_0+b_0$ case of Proposition~\ref{prop:lower}.  Since equality in Theorem~\ref{thm:multiplicity-upper-bound} is achieved for $(a,b)=(c,c)$ by $u,v$, we see that equality is also achieved for all of the other pairs of words considered in the inductive argument of Theorem~\ref{thm:multiplicity-upper-bound} (which can easily be run directly with LCS's rather than passing through SCS's).  For instance, the words $u_{[2,n-a]}, v$ are an equality case of Theorem~\ref{thm:multiplicity-upper-bound} for $(a,b)=(c,c-1)$, and the words $u, v_{[2,n-b]}$ are an equality case of Theorem~\ref{thm:multiplicity-upper-bound} for $(a,b)=(c-1,c)$.  Continuing in this manner, we eventually reach an equality case for $(a,b)=(a_0, b_0)$, as needed.

To prove Proposition~\ref{prop:lower}, we recursively compute $\LCS(u,v)$ and $\ml(u,v)$ for all words $u=(10)^{\langle a \rangle}$, $v=(0110)^{\langle b\rangle}$ with $b$ even.  We introduce the notation $\ell(a,b)\coloneqq\LCS((10)^{\langle a\rangle},(0110)^{\langle b\rangle})$ and $m(a,b)\coloneqq \ml((10)^{\langle a\rangle},(0110)^{\langle b\rangle})$. We begin by computing $\ell(a,b)$.
\begin{lemma}
\label{lem:ell}For $a,b\ge0$ with $b$ even, we have
\[
\ell(a,b)=\begin{cases}
a & \text{if }a\le b/2\\
\frac{b}{2}+\floor{\frac{2a-b}{4}} & \text{if }\frac{b}{2}<a\le\frac{3b}{2}\\
b & \text{if }a>3b/2.
\end{cases}
\]
\end{lemma}

\begin{proof}
The pairs $(a,b)$ with $a\le2$ or $b=0$ can be checked by hand, so we restrict our attention to $a\ge3$ and $b\ge2$.  Note that every LCS of $u,v$ is, according to its first letter, of the form
$$1x \quad \text{or} \quad 01y,$$
where $x$ is an LCS of $(01)^{\langle a-1 \rangle}, (1001)^{\langle b-2 \rangle}$ and $y$ is an LCS of $(01)^{\langle a-3 \rangle}, (1001)^{\langle b-2 \rangle}$.  Exchanging the roles of $0$ and $1$, we find that
$$\LCS((01)^{\langle a-1 \rangle}, (1001)^{\langle b-2 \rangle})=\LCS((10)^{\langle a-1 \rangle}, (0110)^{\langle b-2 \rangle})=\ell(a-1,b-2),$$
and likewise $\LCS((01)^{\langle a-3 \rangle}, (1001)^{\langle b-2 \rangle})=\ell(a-3,b-2)$.  It follows that
$$\ell(a,b)=\max\{1+\ell(a-1,b-2),2+\ell(a-3,b-2)\},$$
and it is not difficult to check that the function defined in the lemma statement is the unique function satisfying this recurrence and the same initial conditions.
\end{proof}

It remains to compute $m(a,b)$.
\begin{lemma}
\label{lem:lower-general}
For $a,b\ge0$ with $b$ even, we have
\[
m(a,b)=\begin{cases}
\binom{b/2}{(2a-b)/4} & \text{if }2a\equiv b\pmod4\\
\binom{b/2+1}{(2a-b+2)/4} & \text{if }2a\equiv b+2\pmod4.
\end{cases}
\]
\end{lemma}

\begin{proof}
As in Lemma~\ref{lem:ell}, we deal separately with the small cases where $a=0$ or $b \leq 2$.  Otherwise, following the same case distinction as in Lemma~\ref{lem:ell}, we find that
$$m(a,b)=m(a-1,b-2) \cdot \mathbbm{1}_{1+\ell(a-1,b-2)\geq 2+\ell(a-3,b-2)}+m(a-3,b-2)\cdot \mathbbm{1}_{1+\ell(a-1,b-2)\leq 2+\ell(a-3,b-2)},$$
where $\mathbbm{1}$ is the $0$-$1$ indicator function of its argument.  Using the exact values of $\ell$ from Lemma~\ref{lem:ell}, we can rewrite this equation as
\[
m(a,b)=\begin{cases}
m(a-1,b-2) & \text{if }a\le b/2\\
m(a-1,b-2)+m(a-3,b-2) & \text{if }b/2<a<\frac{3b}{2}\\
m(a-3,b-2) & \text{if }a\ge\frac{3b}{2},
\end{cases}
\]
and the lemma follows from induction and Pascal's Identity.
\end{proof}

Taking $a=b=4c-2$ in the previous two lemmas gives Proposition~\ref{prop:lower}.

\end{document}